\documentclass[11pt,a4paper]{article}
\usepackage[latin1]{inputenc}
\usepackage[english]{babel}
\usepackage{amsmath}
\usepackage{amsfonts}
\usepackage{amssymb}
\usepackage{amsthm}
\usepackage{graphicx}
\usepackage{appendix}
\usepackage{geometry} 
\usepackage{centernot}
\usepackage{mathtools}
\usepackage{ stmaryrd }
\newtheorem{theorem}{Theorem}[section]
\newtheorem{remark}{Remark}[section]

\newtheorem{definition}{Definition}[section]

\newtheorem{lemma}{Lemma}[section]
\newtheorem{corollary}{Corollary}[section]

\author{Jason Ledwidge \footnote{Eberhard Karls Universit{\"a}t T{\"u}bingen, T{\"u}bingen, Germany. \textit{Email address}: jason.ledwidge@math.uni-tuebingen.de}}
\title{The sharp Li-Yau equality on Shrinking Ricci Solitons with applications}
\date{}
\begin{document}
	\makeatletter
	\maketitle 
	\@author
	\@title 
	\makeatother 
	\begin{abstract}
		We prove that the sharp Li-Yau equality holds for the conjugate heat kernel on shrinking Ricci solitons without any curvature or volume assumptions. This quantity yields several estimates which allows us to classify four dimensional, non-compact shrinking Ricci solitons, which arise as Type I singularity models to the Ricci flow.
	\end{abstract}
	\tableofcontents

\section{Introduction}
R.Hamilton's Ricci flow, 
\begin{equation}\label{Ricci flow}
\partial_{t}g = - 2\mathrm{Ric}_{g}
\end{equation}
can formally be seen as a degenerate nonlinear heat equation for the Riemannian metric, $g,$ in a Harmonic coordinate system. The metric $g$ satisfies the degenerate elliptic system of equations for the metric in this coordinate system
\begin{equation}\label{Harmonic coordinates}
\Delta g = - 2\mathrm{Ric} + Q(g,Dg),
\end{equation}
where the term $Q(g,Dg)$ is quadratic in the covariant derivatives of $g$ and so is of lower order. Hence one can formally write (as in \cite{formation} Introduction)
\begin{equation}
"\partial_{t}g = \Delta g."
\end{equation}
The degeneracy of the Ricci flow stems from the fact that for a diffeomorphism 
\begin{equation*}
\varPsi: N \mapsto N
\end{equation*} 
the curvature tensor is such that (see \cite{Besse} Chapter 5B, equation 5.4)
\begin{equation*}
\varPsi^{*}(\mathrm{Rm}(g)) = (\mathrm{Rm}(\varPsi^{*}g)),
\end{equation*}
where $\mathrm{Rm}$ denotes the Riemannian curvature tensor. Hence \eqref{Ricci flow} is invariant under actions of the diffeomorphism group, which is infinite dimensional. \newline 
Remarkably, G.Perelman who studied the coupled system 
\[    
\begin{cases}
\partial_{t}g = -2\mathrm{Ric} \\
\partial_{t}u = -\Delta u + \mathrm{R}u,
\end{cases}
\] 
where the second equation is the conjugate heat equation and $u$ is the conjugate heat kernel, 
\begin{equation}
u = (-4\pi t)^{-\frac{n}{2}}e^{-f} 
\end{equation}
stated in \cite{RIC} Remark 9.6 that the Ricci flow, a degenerate parabolic system of equations, can be characterised by the short time asymptotics of the conjugate heat kernel, which is the solution of a parabolic scalar equation: \\

"Ricci flow can be characterized among all other evolution equations by the infinitesimal behavior of the fundamental solutions of the conjugate heat equation $\ldots$ Consider the fundamental solution $u = (-4\pi t)^{-\frac{n}{2}}e^{-f} \ldots,$ starting as $\delta$-function at some point $(p,0) \ldots$ The Ricci flow is characterized by the condition $\big(\square \bar{f} + \frac{\bar{f}}{t}\big)(q,t) = o(1),$ in fact, it is $O(|pq|^{2} + |t|)."$ \footnote{ Where $\bar{f} = f - \int f \ u\mathrm{dvol}.$}\\ 

In Euclidean space - a static solution to the Ricci flow - the heat kernel is such that, $f=\frac{|x|^{2}}{4t},$ and $\big(\square \bar{f} + \frac{\bar{f}}{t}\big)(q,t) = 0.$  Hence, the Remark implies that for short times, the conjugate heat kernel along the Ricci flow behaves like the Euclidean heat kernel. The analysis of the logarithm of the heat kernel - i.e. the function $f$ in the context of the Ricci flow - plays a crucial role in geometric analysis, such as in the Li-Yau inequality. \\ 
 
Recall the Li-Yau inequality under non-negative Ricci curvature (see \cite{LY} Theorem 1.3) and let $H$ be the a positive positive solution to the heat equation, $\partial_{t}H = \Delta H.$ Then, 
\begin{equation}\label{LYI}
|\nabla \log H|^{2}-\partial_{t}\log H = -\Delta \log H  \leq \frac{n}{2t}.
\end{equation}
If $H$ is the the heat kernel - which is unique under the curvature assumption - then \eqref{LYI} coupled with Varadhan's short time asymptotic formula, (see \cite{SRS} Theorem 2.2) 
\begin{equation} \label{Varadhan}
\lim_{t \mapsto 0} 4t \log H(x,y,t) = - d^{2}_{g}(x,y),
\end{equation}
implies the Laplacian comparison formula
\begin{equation}
-\lim_{t \mapsto 0}4t\Delta \log H(x,y,t) \leq 2n.
\end{equation}
Thus for static Riemannian manifolds with non-negative Ricci curvature, the short time Euclidean behaviour of the heat kernel encodes strong geometric properties of the manifold. \\

In the context of the Ricci flow on closed manifolds, G.Perelman's monotonicity formula - see \cite{RIC} Section 3; we will introduce this quantity in the next section - is dependent on the conjugate heat kernel and its logarithm. The monotonicity formula expresses that the $\mathcal{W}$ functional is non-decreasing along the Ricci flow coupled to the conjugate heat equation and stationary if $N$ is a shrinking Ricci soliton.\\ 

Ricci solitons are ancient solutions to the Ricci flow modulo a diffeomorphism. More precisely, suppose that $(M,g(\tau))$ is a solution to the Ricci flow with initial condition $g(0)=g_{0}.$ Then the solution $g(\tau)$ is a soliton if there exists a one parameter family of diffeomorphisms, $\varPsi_{\tau},$ such that 
\begin{equation*}
	\varPsi_{\tau} : (M,g_{0}) \longmapsto (M,g(\tau))
\end{equation*}
and
\begin{equation}
	g(\tau) = (1 +\tau)\varPsi^{*}_{\tau} g_{0}, \label{soliton metric}
\end{equation} 
with $\varPsi(0) = \text{Identity}.$ \newline 
Remark 9.6 and Sections 3 and 4 of \cite{RIC} serves as a motivation to study shrinking Ricci solitons whose diffeomorphisms, $\Psi_{\tau},$ are generated by the logarithm of the conjugate heat kernel. From henceforth, a shrinking soliton will always be a complete, gradient soliton whith the potential function $f$ coming from the conjugate heat kernel, 
\begin{equation}
	u = (4\pi \tau)^{-n/2}e^{-f}.
\end{equation} 
We will assume that $u$ is non-trivial, i.e. is not a constant. \newline
Taking the time derivative of \eqref{soliton metric} at $\tau = 0$ yields, the shrinking soliton equation 
\begin{equation}
	\mathrm{Ric} + \nabla^{2}f = \frac{g}{2\tau}, \label{soliton equation}
\end{equation} 
and tracing \eqref{soliton equation}, one has 
\begin{equation}
	\mathrm{R} + \Delta f = \frac{n}{2\tau}. \label{trace}
\end{equation} 
An important quantity is that the gradient of the scalar curvature is such that
\begin{equation}\label{scalar gradient}
	\nabla\mathrm{R} = 2\mathrm{Ric}(\nabla f, \cdot).
\end{equation}
This follows from applying the gradient operator and the second contracted Bianchi identity to \eqref{soliton equation}.\newline 
Since complete shrinking solitons are ancient solutions and hence have non-negative scalar curvature. In the compact case, this follows directly from the maximum principle applied to the evolution for the scalar curvature, 
\begin{equation*}
	\partial_{\tau} \mathrm{R}= -(\Delta \mathrm{R} + 2|\mathrm{Ric}|^{2}).
\end{equation*} 
Then by the Cauchy-Schwarz Trace inequality and standard PDE theory, $\mathrm{R} \geq -\frac{n}{2\tau}.$ See \cite{RIC} page 18 for instance. As the solution exists on $(0, \infty)$ we send $\tau \mapsto \infty$ to get the desired result. For the non-compact case without assumptions on bounded geometry, see \cite{CHEN} Corollary 2.5. \\

The importance of shrinking Ricci solitons is that they arise as the singularity models to Type I solution to the Ricci flow. The model example of a Type I singularity is the neck pinch.
\begin{definition}
	$(M,g(t)),$ is a Type I solution to the Ricci flow, $\partial_{t}g=-2\mathrm{Ric},$ if it exists on a finite time interval $(0, T],$ such that there exists a positive constant, $\tilde{C},$ such that 
	\begin{equation*}
		\lim_{t \mapsto T}|\mathrm{Rm}|_{g(\tau)} \leq \frac{\hat{C}}{T-t} < \infty.
	\end{equation*}
\end{definition}
The model example for a Type I singularity is the neck pinch. See \cite{formation} Section 3 for images.\newline
Building upon \cite{Nab} Theorem 1.5, it was shown in \cite{EMT} Theorem 1.1 that a Type I solution to the Ricci flow converges to a non-trivial, canonical gradient shrinking Ricci soliton -meaning that it arises as the minimiser of the $\mathcal{W}$ functional, see \cite{EMT} Definition 2.1 - 
\begin{equation*}
	\mathrm{Ric} + \nabla^{2}f = \frac{g}{2}. 
\end{equation*} 
A normalised soliton \footnote{In the current literature, normalised shrinking solitons are studied via the use of the Bakry-Emery comparison geometry and elliptic PDE theory. Hence our study of $(M,g,f)$ via the analysis of the conjugate heat kernel is a new perspective.} is where one fixes $\tau.$  
\begin{remark}
In fact, it was shown in \cite{CaoZh} Theorem 4.1 - which extends to the non-compact case for Type I $\kappa$ solutions - and \cite{CM} Theorem 1.4 that one can obtain \cite{EMT} Theorem 1.1, by using analysing the blow-up of the Perelman's $\mathcal{W}$ functional. Note however, that if the limit is non-compact, then this functional is ill-defined if the conjugate heat kernel is not unique. We introduce the $\mathcal{W}$ functional in the next section and discuss the issue of uniqueness in Section 3.
\end{remark}
We also note that in 4 dimensions, the canonical shrinking Ricci soliton arises as the singularity model for closed solutions to the Ricci flow which satisfy the curvature condition,  
\begin{equation*}
	\lim_{t \mapsto T}\mathrm{R}( \cdot, t) \leq \frac{\bar{C}}{T-t} < \infty.
\end{equation*}
See the proof of \cite{Bamler} Theorem 1.2. \\ 

The main result is Corollary \ref{4D classification} which shows that the singularities of $(M^{4},g,f)$ are modelled on 
\begin{equation*}
	\mathbb{R}^{4}/SO(4),
\end{equation*} 
without any curvature or volume assumptions. For non-compact solitons in four dimensions, under the assumption of bounded, non-negative curvature operator, it was proved in \cite{Nab} Main Theorem, that these solitons are isometric to either $\mathbb{R}^{4},$ or metric quotients of
\begin{equation*}
\mathbb{R}^{3} \times \mathbb{S}^{1}, \ \mathbb{R}^{2} \times \mathbb{S}^{2}.
\end{equation*}
In order to arrive at this Corollary, we first prove Theorem \ref{kappa non-collapsed}, which shows that $(M,g,f)$ is $\kappa$ non-collapsed at all scales without any assumptions other than the completeness of the soliton. To prove the Theorem, we must first show that $\mathcal{W}$ has a minimiser in the non-compact case which is achieved by proving that the conjugate heat kernel is unique. This is Theorem \ref{uniqueness}. We also prove that the Laplacian comparison formula holds for short times on $(M,g,f)$ in Corollary \ref{Varadhan-Laplace}. \newline 
The aforementioned results follow from a sharp Li-Yau equality (Theorem \ref{Li-Yau}) which holds without any curvature or geometric assumptions. 
 
\section*{Acknowledgements} 
The author is extremely grateful to his supervisor, Professor Gerhard Huisken, for his continuous support and encouragements during the author's PhD Thesis. We would like to thank Martin Kell for pointing out some instances where our work was an improvement on the current literature. The author is also very appreciative to Professors Dominique Bakry, Klaus Ecker and Michel Ledoux for their very early interest in out work. \newline 
The author was partially funded by the Deutsche Forschungsgemeinschaft (DFG, German Research Foundation) - Projektnummer 396662902 - during this work. 
\section{Preliminaries}
\subsection{The heat and conjugate heat kernels} 
In this section, we will present some key properties of the heat kernel which will be key for the understanding of the $\mathcal{W}$ functional, which will be introduced in the next subsection.
\begin{definition}\label{Heat kernel definition}
The heat kernel, $H,$ on a smooth Riemannian manifold $N,$ is the fundamental solution to the parabolic equation, 
\begin{equation}
\partial_{t}H = LH, 
\end{equation}
where $A$ is an elliptic operator, and satisfies the following properties $\forall \ x,y,z \in N, \ t >0:$ \newline 
\bfseries{Symmetry}:
\begin{equation*}
 H(x,y,t) = H(y,x,t)
\end{equation*} 
Non-negativity 
\begin{equation*}
H(x,y,t) \geq 0
\end{equation*}
The integral kernel:
\begin{equation*}
P_{t}h(x) = e^{-tL}h(x) = \int_{N}H(x,y,t)h(y)\mathrm{dvol}(y), \ h \in L^{2}
\end{equation*} 
The Markovian property: 
\begin{equation*}
\int_{N}H(x,y,t)\mathrm{dvol}(y) \leq 1
\end{equation*}
The semigroup property: 
\begin{equation*}
H(x,y, t + s) = \int_{N}H(x,z,t)H(x,z,t)\mathrm{dvol}(z).
\end{equation*}
\end{definition}
If the heat kernel is unique, then this is equivalent to stochastic completeness, which is
\begin{equation}
\int_{N}H(x,y,t)\mathrm{dvol}(y) = 1.
\end{equation}
On a compact Riemannian manifold, the heat kernel with $L = \Delta,$ is always unique. However, in the non-compact setting, this is no longer true without certain curvature or geometric conditions. For instance, if the Ricci curvature is non-negative, then the heat kernel is unique. \newline
Along the Ricci flow, the volume form is evolving and hence the conjugate heat equation is not just the backwards heat equation. Suppose that $N$ is a closed manifold or a manifold where integration by parts is well defined.\footnote{The conjugate heat operator can be defined independently of this integral formula and hence is well defined on non-compact manifolds, however we have included it just for demonstration.} We define the conjugate heat operator as follows: let $\varphi$ be a solution to the heat equation and let $\psi \in L^{1}.$ Then, since the volume form evolves along the Ricci flow by, 
\begin{equation*}
\begin{split}
\partial_{t}\mathrm{vol} &= \partial_{t}(\mathrm{det}g)^{1/2} \\
& = \frac{1}{2}\mathrm{Trace}(\partial_{t}g) \\
&= -\mathrm{Rvol},
\end{split}
\end{equation*}
we see that
\begin{equation}
\begin{split}
\int (\partial_{t} - \Delta)\varphi.\psi\mathrm{dvol} & = \int \big( -\partial_{t}(\psi\mathrm{dvol}) + \Delta \psi \big) \ \varphi\mathrm{dvol} \\ 
& = \int (-\partial_{t}\psi + \mathrm{R}\psi)\ \varphi\mathrm{dvol}.
\end{split}
\end{equation}
Hence, the conjugate heat operator is define by 
\begin{equation}
-\partial_{t} + \Delta - \mathrm{R},
\end{equation}
i.e. $L = \Delta - \mathrm{R}$ in Definition \ref{Heat kernel definition}. \newline 
The explicit formula for the conjugate heat kernel along the Ricci flow, 
\begin{equation*}
u = (4\pi \tau)^{-n/2}e^{-f},
\end{equation*}
implies a formula for the function $f,$ 
\begin{equation}\label{f equation} 
\partial_{t}f= - \Delta f + |\nabla f|^{2} - \mathrm{R} + \frac{n}{2\tau}.
\end{equation}
\subsection{The $\mathcal{W}$ functional}
Of central importance to G.Perelman's solution to the Poincare and geometrisation conjectures was his introduction of the $\mathcal{W}$ functional for closed solutions to the Ricci flow in \cite{RIC} Section 3.\newline 
The $\mathcal{W}$ functional is defined on a compact manifold for $\tau = T -t >0,$ where $T$ is the final time of existence, as
\begin{equation}\label{W functional}
\mathcal{W}(f,g,\tau) = \bigg\{\int_{N} [\tau(|\nabla f|^{2} + \mathrm{R}) + f -n] \ u\text{dvol}_{g(t)} \ : \int_{N} u\text{dvol}_{g(t)} =1\bigg\},
\end{equation}
where 
\begin{equation*}
u = (4\pi \tau)^{-n/2}e^{-f}
\end{equation*}
is the conjugate heat kernel. For the Euclidean heat kernel, $(4\pi \tau)^{-n/2}e^{-\frac{|x|^{2}}{4\tau}},$ if we set $\tau = \frac{1}{2},$ then 
\begin{equation*}
\begin{split}
\mathcal{W}(f,g,1/2) & = (2\pi)^{-n/2}\int_{\mathrm{R}^{n}} \bigg(\frac{1}{2}|\nabla f|^{2} + f - n\bigg) \ e^{-f}\text{dx} \\ 
& = 0,
\end{split}
\end{equation*}
where the second inequality follows from integration by parts with respect to the weighted measure. The above formula is known as the Euclidean logarithmic Sobolev equality and is equivalent to L.Gross' Gaussian logarithmic Sobolev inequality
\begin{equation}\label{Gross}
(2\pi)^{-n/2}\int_{\mathbb{R}^{n}} \big(2|\nabla \log \Phi|^{2} - \log\Phi^{2}\big) \ \Psi^{2}e^{-\frac{|x|^{2}}{2}}\mathrm{dx} \geq 0.
\end{equation}
To see this, let $\Psi = e^{\frac{|x|^{2}}{4} -\frac{f}{2}}.$ 
\begin{remark}
Despite the equivalence between the Gaussian logarithmic Sobolev inequality and its Euclidean counterpart, the inequalities tell us vastly different things. \newline 
The Gaussian logarithmic Sobolev inequality is equivalent to E.Nelson's hypercontractivity estimate for the Orstein-Uhlenbeck semigroup\footnote{This semigroup is the integral representation to the fundamental solution to the heat equation with a drift, where the diffusion operator is $L = \partial_{t} - \Delta - x.\nabla.$}, $\mathrm{O}_{t}$ 
\begin{equation*}
||\mathrm{O}_{t}h||_{L^{4}} \leq c||h||_{L^{2}}, \ c>0.
\end{equation*}
The Euclidean version is equivalent (up to constants) to the Sobolev inequality.   See \cite{BGL} Proposition 6.2.3.
The Sobolev inequality is equivalent to ultracontractivity of the heat semigroup (see previous subsection or \cite{BGL} Theorem 6.3.1), $\mathrm{P}_{t}$ 
\begin{equation*}
\|\mathrm{P}_{t}h\|_{L^{\infty}} \leq c\|h\|_{L^{1}}, \ c>0.
\end{equation*}
Hence the Euclidean logarithmic Sobolev inequality is equivalent to a maximal regularity estimate of the heat semigroup, whereas the Gaussian version improves the regularity of the Ornstein-Uhlenbeck semigroup, but one cannot obtain maximal regularity. 
\end{remark}
Consequently, the $\mathcal{W}$ functional can be viewed as a generalisation of the Euclidean logarithmic Sobolev inequality along the Ricci flow coupled to the conjugate heat equation. More precisely, if we define the scaled Boltzmann Entropy term\footnote{The $\int\varPhi^{2}\log \varPhi^{2}$ term in the Gaussian logarithmic Sobolev inequality is known as the $L^{2}$ Boltzmann Entropy.} - In the Ricci flow literature, this is often called the Nash Entropy - 
\begin{equation*}
\begin{split}
N(u) &= - \int_{N}\bigg(f - \frac{n}{2} \bigg) \ u\text{dvol}_{g(\tau)} \\
& = \int_{N} \bigg[\log\big((4\pi\tau)^{n/2}u\big) + \frac{n}{2}\bigg]\ u\text{dvol}_{g(\tau)} 
\end{split}
\end{equation*}
where $u$ is the conjugate heat kernel. Multiplying by $\tau$ and applying the backwards heat operator to this functional yields \eqref{W functional}, i.e.
\begin{equation*}
\begin{split}
(\partial_{t}+ \Delta) \tau N(u) & = \int_{N} \bigg[\tau\bigg(\Delta f - |\nabla f|^{2} + \mathrm{R} - \frac{n}{2\tau}\bigg) + \tau\Delta f + f - \frac{n}{2} \bigg] \ u\mathrm{dvol}_{g(\tau)} \\ 
& = \int_{N} \bigg[\tau\big(2\Delta f - |\nabla f|^{2} + \mathrm{R} \big) + f - \frac{n}{2} - \frac{n}{2} \bigg] \ u\mathrm{dvol}_{g(\tau)} \\ 
& = \mathcal{W}(f,g,\tau),
\end{split} 
\end{equation*}
where the last lines follows from integration by parts on the heat kernel measure. \newline
The next Theorem can be found in \cite{RIC} Theorem 3.2 and demonstrates the importance of this functional. 
\begin{theorem}
The first variation of \eqref{W functional} is such that 
\begin{equation}\label{first variation} 
\delta{W}(f,g,\tau) = 2\tau\ \int_{N} \bigg|\mathrm{Ric} + \nabla^{2}f - \frac{g}{2\tau}\bigg|^{2}  \ u\text{dvol}_{g(t)}.
\end{equation}
\end{theorem}
Thus the $\mathcal{W}$ functional is monotonically non-decreasing along the Ricci flow and is stationary if and only if the solution is a shrinking Ricci soliton - which we write as $(M,g,f)$ - i.e. 
\begin{equation}
\mathrm{Ric} + \nabla^{2}f = \frac{g}{2\tau}.
\end{equation}
The model shrinking Ricci soliton is the Gaussian soliton, which is Euclidean space with the Euclidean heat kernel. That is, $(\mathbb{R}^{n},|\cdot|, \frac{|x|^{2}}{4\tau}).$ With this in mind, one can view the value of the minimiser of \eqref{W functional}, 
\begin{equation}\label{minimiser}
\mu(g,\tau) = \inf_{f}\mathcal{W}(f,g,\tau),
\end{equation}
as an indicator of how 'close' the solution is to Euclidean space.  We will discuss shrinking Ricci solitons in greater detail in the next Section.
\begin{remark}
On a manifold with non-negative Ricci curvature, if $\mathcal{W}(f,g,1/2) \geq 0,$ then the manifold is isometric to Euclidean space. See \cite{BCL} Corollary 1.6.
\end{remark} 
By the uniqueness of the conjugate heat kernel on a compact manifold, $\mu(g,\tau)$ always exists.\newline 
We note that \eqref{first variation} gives the Ricci flow (modulo a diffeomorphism) a gradient flow structure. This is a desirable property since the Ricci flow is only weakly parabolic due to the invariance of the equation under the diffeomorphism group, which is infinite dimensional. Hence \eqref{W functional} allows one to break the diffeomorphism invariance and to study the Ricci flow as a strongly parabolic system. Furthermore, \eqref{first variation} expresses that the Boltzmann Entropy functional is convex along the Ricci flow coupled to the conjugate heat equation. This is analogous to the fact that the Boltzmann/Nash Entropy is convex along the heat flow for manifolds with non-negative Ricci curvature. See \cite{VRSturm} Theorem 1.1. \\

The existence of \eqref{minimiser} was used by G.Perelman to prove that solutions to the Ricci flow with a finite time of existence on a closed manifold, is $\kappa$ non-collapsed - \cite{RIC} Theorem 4.1 - which is defined as follows (see \cite{RIC} Definition 4.2).
\begin{definition}
A metric $g$ is $\kappa$ non-collapsed on the scale $\rho$ if for every metric ball of radius $ r < \rho$ with $|\mathrm{Rm}|_{g} \leq r^{-2},$ is such that
\begin{equation}
\mathrm{vol}_{g}\big(B(x,r)\big) \geq \kappa r^{n}.
\end{equation}
\end{definition}
The proof of \cite{RIC} Theorem 4.1 is by contradicting the existence of a unique minimiser for \eqref{W functional}. The proof of the $\kappa$ non-collapsed condition for closed solutions to the Ricci flow implies an injectivity radius lower bound at finite scales, 
\begin{equation*}
\mathrm{inj} \geq \rho r >0.
\end{equation*}
See \cite{CGT} Theorem 4.7 for precise details and links to the analysis of the heat kernel on static manifolds. 

\section{The structure of shrinking solitons } 
Recall that $f$ is a minimiser of \eqref{W functional} if, $u = (4\pi \tau)^{-n/2}e^{-f},$ is the conjugate heat kernel. The conjugate heat kernel approach gives us a direct way to compare $(M,g,f)$ with the model Gaussian shrinking soliton.
The Gaussian soliton is optimal in the sense that it has the following properties; \newline
 $\mathcal{W}(\frac{|x|^{2}}{4\tau},|\cdot|,1/2)= 0;$ it optimises Hamilton's auxiliary equation \eqref{auxiliary}; it maximises the reduced volume. We wish to show that $(M,g,f),$ is the best competitor for the Gaussian soliton. \newline
We will indicate whether we are working with compact or non-compact $(M,g,f).$ If no indication is given, then this means that the result holds in both settings. To our knowledge, all results on the conjugate heat kernel are new. \\ 

Note that one can view the Ricci flow on $(M,g,f)$ from an almost purely analytic point of view. Writing \eqref{soliton equation} in terms of the conjugate heat kernel yields, 
\begin{equation*}
\nabla^{2} \log u + \frac{g}{2\tau} = \mathrm{Ric}.
\end{equation*}
If the Ricci tensor were non-negative, then the shrinking soliton equation would say that 'along the Ricci flow, the conjugate heat kernel is log-convex.' In addition, the Ricci flow on $(M,g,f)$ can be written as 
\begin{equation*}
\partial_{\tau}g = \nabla^{2} \log u + \frac{g}{2\tau}.
\end{equation*}
Therefore, the Ricci flow on $(M,g,f)$ says the following; if the Ricci tensor is bounded - which one can show is true on a small ball by combining Theorem \ref{uniqueness} and \cite {HM2} Theorem 1.1 - then the evolution of the metric tensor is equivalent to an error estimate on the log convexity of the conjugate heat kernel. \\ 

In this section, we will study the coupled system 
\[    
\begin{cases}
\partial_{\tau}g = 2\mathrm{Ric} \\
\partial_{\tau}u = \Delta u - \mathrm{R}u. 
\end{cases}
\]  \\

The aim of this section is to prove that the conjugate heat kernel, $u = (4\pi \tau)^{-n/2}e^{-f},$ is unique, which will imply that the $\mathcal{W}$ functional has a unique minimiser. In the current literature, a minimiser for the $\mathcal{W}$ functional is always an assumption. See \cite{HM} Theorem 1.1 or \cite{CarLi} Theorem 1.1 for example. \newline 
In the static case, uniqueness of the heat kernel is known under the assumption of quadratic exponential growth of geodesic balls, i.e. $\text{vol}(B(x,r)) \leq e^{Ar^{2}}, \ A >0.$\footnote{The volume bound is optimal if one considers that uniqueness of a solution, $\psi,$ to a diffusion equation in Euclidean space must satisfy, $|\psi|\leq e^{A|x|^{}2}.$} This can be seen in \cite{KarpLi} Theorem 1, whose proof we shall follow. The strategy of the proof in \cite{KarpLi} involves integration by parts on, 
\begin{equation}
0 = 2 \int_{0}^{T}\int_{N} \phi^{2}e^{\varPhi}(\Delta_{y} - \partial_{\tau})\hat{u}(y,t)\ \text{dvol}_{g}(y)dt, 
\end{equation} 
where 
\begin{equation*}
\varPhi(x,y,s) = \frac{-d_{g}^{2}(x,y)}{4(2T-s)}, \ \ 0 \leq s <T,
\end{equation*}
\begin{equation*}
\varphi(y) = \varphi(d^{2}(x,y)) = \begin{cases} 
1 \ \text{on} \ B(x,r)  \\
0 \ \text{on} \ B(x,r + \epsilon),  
\end{cases} 
\end{equation*}
and the use of Moser's iteration.\newline 
For the conjugate heat kernel on $(M,g,f),$ in addition to the volume growth condition, we need an upper bound on the scalar curvature as we must perform integration by parts on 
\begin{equation}
0 = 2 \int_{0}^{\hat{\tau}}\int_{M} \phi^{2}e^{\varPhi}(\Delta_{y} + \mathrm{R} - \partial_{\tau})u(y,t)\ \text{dvol}_{g}(y)d\tau. 
\end{equation} 
Both of these required conditions - volume and scalar curvature bounds - will follow from the proof of the sharp Li-Yau equality for the conjugate heat kernel on $(M,g,f).$
We now show that the fundamental solution of the conjugate heat kernel on shrinking solitons yields the sharp form of the Li-Yau inequality stated in the introduction.\newline 
Recall that to prove the Li-Yau inequality on a static Riemannian manifolds, one must take the time derivative of $\log h$ where $h$ is a solution to the heat equation. See \cite{LY} Theorem 1.1. For the conjugate heat kernel along the Ricci flow, we have an explicit formula for $\log u.$  
\begin{theorem}\label{Li-Yau} 
The conjugate heat kernel, $u,$ on $(M,g,f)$ satisfies the sharp Li-Yau equality, 
\begin{equation}
|\nabla \log u|^{2} - \partial_{\tau}\log u = \frac{n}{2\tau}. 
\end{equation} 
\end{theorem}
As a result of \eqref{trace}, on a shrinking Ricci soliton, the potential function solves a Hamilton-Jacobi equation, (with a convex, quadratic Hamiltonian, see \cite{BGL} Chapter 9.4)
\begin{equation} \label{f equation soliton}
\begin{split}
\partial_{\tau}f &= \Delta f -|\nabla f|^{2} + \mathrm{R} - \frac{n}{2\tau} \\ 
& = -|\nabla f|^{2}. 
\end{split}
\end{equation}
The equation holds for the Euclidean heat kernel with, $f= \frac{|x|^{2}}{4\tau}.$ 
\begin{proof} 
Since, $\log u = - (f+\frac{n}{2}\log(4\pi \tau)),$ 
\begin{equation*}
\begin{split}
\partial_{\tau} \log u &= -\partial_{\tau}\big(f+\frac{n}{2}\log(4\pi \tau)\big) \\
&= -\big(-|\nabla f|^{2}  +\frac{n}{2\tau}\big) \\
& = |\nabla f|^{2} -\frac{n}{2\tau}.
\end{split}
\end{equation*}
As $\nabla \log u = -\nabla f,$ we obtain the sharp Li-Yau type equality. 
\end{proof} 
The sharpness follows from the fact that we have equality for the Euclidean heat kernel i.e. the Gaussian shrinking soliton. 
\begin{remark}
The Li-Yau equality is essential for the rest of this section. Consequently, the results that follow do not hold for steady or expanding solitons. 
\end{remark} 
We now prove the necessary volume bound via the integrated Li-Yau inequality. 
\begin{corollary}
On, $(M,g,f),$ the conjugate heat kernel satisfies the differential Harnack inequality, 
\begin{equation}
u(x_{2}, \tau_{2}) \geq u(x_{1}, \tau_{1}) \bigg(\frac{\tau_{1}}{\tau_{2}}\bigg)^{n/2}e^{-\frac{d^{2}_{g(\tau)}(x_{1},x_{2})}{4(\tau_{2} - \tau_{1})}}. \label{Harnack}
\end{equation}
\end{corollary}
\begin{proof}
Let $\gamma(\tau),$ be a minimising geodesic i.e. , $\gamma,$ is such that, $\int_{\tau_{1}}^{\tau_{2}}|\dot{\gamma}|^{2} d\tau = (\tau_{2} - \tau_{1})\int_{0}^{1}|\dot{\gamma}|^{2} d\tau.$ Then, 
\begin{equation*}
\begin{split} 
\partial_{\tau} \log u (\gamma(\tau), \tau) & =  \partial_{\tau}\log u + \langle \nabla \log u, \dot{\gamma} \rangle  \\
&= |\nabla \log u|^{2} - \frac{n}{2\tau} + \langle \nabla \log u, \dot{\gamma} \rangle  \\ 
& = \big|\nabla \log u - \frac{\dot{\gamma}}{2} \big|^{2} - \frac{|\dot{\gamma}|^{2}}{4} - \frac{n}{2\tau} \\ 
& \geq - \bigg(\frac{n}{2\tau} + \frac{|\dot{\gamma}|^{2}}{4} \bigg).
\end{split}
\end{equation*} 
Therefore,
\begin{equation*}
\begin{split}
\log\bigg(\frac{u(\gamma(\tau_{2}), \tau_{2})}{u(\gamma(\tau_{1}), \tau_{1})}\bigg) & = \int_{\tau_{1}}^{\tau_{2}} \partial_{\tau} \log u(\gamma(\tau), \tau) \ d\tau \\
& \geq - \int_{\tau_{1}}^{\tau_{2}} \bigg(\frac{n}{2\tau} + \frac{|\dot{\gamma}|^{2}}{4} \bigg) \ d\tau.
\end{split} 
\end{equation*}
Taking exponentials and setting $\gamma(\tau_{i}) = x_{i},$ 
\begin{equation*}
\frac{u(x_{2}, \tau_{2})}{u(x_{1}, \tau_{1})} \geq \bigg(\frac{\tau_{2}}{\tau_{1}}\bigg)^{-n/2}\exp\bigg(-\frac{1}{4} \int_{\tau_{1}}^{\tau_{2}} |\dot{\gamma}|^{2} \ d\tau \bigg).
\end{equation*} 
As $\gamma$ is a minimising geodesic, the result follows.
\end{proof}
Since $u$ is a Dirac mass along the diagonal as $\tau_{1} \mapsto 0$ (see \cite{RIC} Corollary 9.3, with $T = 0$), \newline 
\begin{equation*}
\lim_{\tau_{1} \mapsto 0} (4\pi \tau_{1})^{n/2}u(x,x,\tau_{1})=1.
\end{equation*}
Hence we obtain a lower bound on the conjugate heat kernel, 
\begin{equation}
u(x,y,\tau) \geq (4\pi \tau)^{-n/2}e^{-\frac{d^{2}_{g(\tau)}(x,y)}{4\tau}}. \label{HKLB} 
\end{equation}
An immediate consequence of this lower bound is a Bishop-Gromov inequality which can be found in \cite{Chavel} Theorem 8.14 or \cite{GHL} Theorem 3.2. Such an inequality is known to hold for normalised shrinking solitons. See \cite{HM} Lemma 2.2.
\begin{corollary}
By \eqref{HKLB}, for all $\tau >0,$ we have the following volume growth bound, 
\begin{equation}
\frac{\mathrm{vol}_{g(\tau)}\big(B(x,r)\big)}{r^{n}} \leq K(4\pi)^{n/2}, \ K > 0 \label{Bishop-Gromov}.
\end{equation}
\end{corollary} 
\begin{proof}
By the definition of the heat kernel, 
\begin{equation*}
\begin{split}
1 & \geq \int_{M} u \ \text{dvol}_{g(\tau)} \\ 
& \geq (4\pi \tau)^{-n/2}\int_{M}e^{-\frac{d^{2}_{g(\tau)}(x,y)}{4\tau}} \ \text{dvol}_{g(\tau)} \\
& \geq (4\pi \tau)^{-n/2}\int_{B(x,r)}e^{-\frac{d^{2}_{g(\tau)}(x,y)}{4\tau}} \ \text{dvol}_{g(\tau)} \\
& \geq e^{-k}(4\pi \tau)^{-n/2}\int_{B(x,r)} \text{dvol}_{g(\tau)} \\
& \geq  e^{-k}(4\pi \tau)^{-n/2}\mathrm{vol}_{g(\tau)}(B(x,r)).
\end{split}
\end{equation*}
By parabolic scaling, $\tau^{-n/2} \sim r^{-n}.$ 
\end{proof}
\begin{remark}
The volume bound \eqref{Bishop-Gromov} is independent of curvature bounds, which is not the case in the present literature. For instance, if $\mathrm{Ric}\geq \delta >0$ as in \cite{CarLi} Corollary 2.2, then the volume growth is such that 
\begin{equation*}
\mathrm{vol}\big(B(x,r)\big) \leq r^{n - 2\delta}, \ \delta = \delta(M,f) >0.
\end{equation*}
\end{remark}
The above Corollary implies the volume bound, $\text{vol}(B(x,r)) \leq e^{Ar^{2}}, \ A >0.$ \newline 
Next, we prove the scalar curvature bound. To do so, we recall an important equation on shrinking solitons, first seen in \cite{formation} Theorem 20.1, 
\begin{equation}
R + |\nabla f|^{2} - \frac{f}{\tau} = \Lambda(\tau) \label{auxiliary}
\end{equation}
where $\Lambda(\tau)$ is a constant. For the Euclidean heat kernel, $\Lambda(\tau) = 0.$ 
\begin{lemma}
The scalar curvature on $(M,g,f)$ is such that 
\begin{equation}
0 \leq \mathrm{R} \leq \Lambda(\tau) + \frac{d^{2}_{g(\tau)}(x,y)}{4\tau^{2}}, \label{scalar bounds}
\end{equation} 
\end{lemma}
\begin{proof}
The lower bound follows from \cite{CHEN} Corollary 2.5. For the upper bound, if we write in terms of the potential $f,$ the differential Harnack inequality expresses that, 
\begin{equation}
f(x_{2}, \tau_{2}) - f(x_{1}, \tau_{1}) \leq \frac{d^{2}_{g(\tau)}(x_{2},x_{1})}{4(\tau_{2} - \tau_{1})} \ . \label{LYc1}
\end{equation}
Letting $\tau_{1} \mapsto 0,$ 
\begin{equation}
f(y,\tau) \leq \frac{d^{2}_{g(\tau)}(x,y)}{4\tau}. \label{LYc2}
\end{equation}
As $|\nabla f|^{2} \geq 0,$ the non-negativity of the scalar curvature and \eqref{LYc2} imply that,
\begin{equation}
0 \leq \mathrm{R} \leq \Lambda(\tau) + \frac{d^{2}_{g(\tau)}(x,y)}{4\tau^{2}}. 
\end{equation}
\end{proof}
By parabolic scaling, the scalar curvature scales correctly since, $\frac{r^{2}}{\tau^{2}} \sim \frac{1}{r^{2}}.$ \newline 
It is listed as an open problem (\cite{RF4} Problem 27.17) to prove that on a normalised shrinking Ricci soliton, 
\begin{equation*}
\mathrm{R} \leq \frac{1}{4 + \varepsilon}(d(p,q) + B)^{2}, \ B > 0.
\end{equation*}
The importance of this bound is that it implies finite topology for normalised shrinking Ricci solitons. This follows from \cite{FMZ} Theorem 1.2 and \cite{RF4} Corollary 27.16. So if we were to set $ \tau = (1 + \varepsilon)^{1/2},$ in \eqref{scalar bounds}, then we obtain the desired bound.

\begin{remark}
In the case of normalised shrinking solitons, the potential is such that 
\begin{equation*}
\frac{1}{4}\big(d(x,p) - 5n \big)^{2} \leq f(x) \leq \frac{1}{4}\big(d(x,p) + \sqrt{2n} \big)^{2}, 
\end{equation*}
where, $p,$ is a point where, $f,$ attains it's maximum. See \cite{HM} Lemma 2.1. This two sided bound is needed to prove \cite{HM} Lemma 2.2. 
\end{remark}
We now have the necessary bounds to prove the uniqueness of the conjugate heat kernel. The proof is now almost identical to \cite{KarpLi} Theorem 1.
\begin{theorem}\label{uniqueness} 
	The conjugate heat kernel, $u,$ on $(M,g,f),$ is uniquely determined by it's initial condition, $u(x,0) = u_{0}(x).$ 
\end{theorem}
\begin{proof}
	It is enough to prove that $u(x,t) = 0,$ if $u_{0}(x) = 0.$ \newline 
	Let, 
	\begin{equation}\label{Phi equation}
	\varPhi(x,y,s) = \frac{-d_{g(0)}^{2}(x,y)}{4(2\bar{\tau}-s)}, \ \ 0 \leq s < \bar{\tau},
	\end{equation}
	where, $\bar{\tau},$ is fixed. Then, 
	\begin{equation}
	\partial_{s}\varPhi = - |\nabla \varPhi|^{2}. \label{distance evolution}
	\end{equation} 
	(Note that, $f,$ solves the above equation.) Define the cut-off function, $\phi(y) = \phi(d^{2}_{g(0)}(x,y)),$ by
	\begin{equation*}
	\phi(y) = \begin{cases} 
	1 \ \text{on} \ B(x,r)  \\
	0 \ \text{on} \ B(x,r + \epsilon),  
	\end{cases} 
	\end{equation*}
	with, $|\nabla \phi| \leq \frac{3}{\epsilon}.$ Now, 
	\begin{equation}
	0 = 2 \int_{0}^{\hat{\tau}}\int_{M} \phi^{2}e^{\varPhi}(\Delta_{y} - \partial_{\tau} - \mathrm{R})u(y,t)\ \mathrm{dvol}_{g(\tau)}(y)d\tau. \label{0 integral}
	\end{equation}
	That the integral, $\int_{0}^{\hat{\tau}}\int_{M} Ru\mathrm{dvol}_{g(\tau)}(y)d\tau$ is finite follows from \eqref{scalar bounds} since $\int_{0}^{\infty} r^{2}e^{-r^{2}} \ dr < \infty$ as $r \mapsto \infty$ and $\int_{\tau}^{\hat{\tau}} s^{-2}e^{-1/s} \ ds < \infty$ for $\tau \mapsto 0.$ \newline
	Integration by parts on \eqref{0 integral} yields, 
	\begin{equation}
	\begin{split}
	0 &= -4 \int_{0}^{\hat{\tau}}\int_{M} \phi e^{\varPhi} \langle \nabla \phi, \nabla u \rangle u\mathrm{dvol}_{g(\tau)}(y)d\tau - 2 \int_{0}^{\hat{\tau}}\int_{M} \phi^{2}e^{\varPhi}\langle \nabla \varPhi, \nabla u\rangle u\mathrm{dvol}_{g(\tau)}(y)d\tau \\ 
	& -2 \int_{0}^{\hat{\tau}}\int_{M} \phi^{2}e^{\varPhi}|\nabla u|^{2}\mathrm{dvol}_{g(\tau)}(y)d\tau - 2\int_{0}^{\hat{\tau}}\int_{M} \phi^{2}e^{\varPhi}u^{2}\mathrm{R}\mathrm{dvol}_{g(\tau)}(y)d\tau \\ 
	& + \int_{0}^{\hat{\tau}}\int_{M}\phi^{2}e^{\varPhi}u^{2}\partial_{s}\mathrm{dvol}_{g(\tau)}(y)d\tau + \int_{0}^{\hat{\tau}}\int_{M} \phi^{2}e^{\varPhi}u^{2}\partial_{s}\varPhi\mathrm{dvol}_{g(\tau)}(y)d\tau  \\ 
	& +\int_{M}\phi^{2}e^{\varPhi}u^{2}|_{s = \hat{\tau}}\mathrm{dvol}_{g(\tau)}(y) - \int_{M}\phi^{2}e^{\varPhi}u^{2}|_{s = 0}\mathrm{dvol}_{g(\tau)}(y).
	\end{split}
	\end{equation}
	The finiteness of $\int_{M} R \ u dvol_{g(\tau)}(y)$ coupled with the evolution equation for the volume form, reduces the equation to,
	\begin{equation}
	\begin{split}
	0 &= -4 \int_{0}^{\hat{\tau}}\int_{M} \phi e^{\Phi} \langle \nabla \phi, \nabla u \rangle u\mathrm{dvol}_{g(\tau)}(y)d\tau - 2 \int_{0}^{\hat{\tau}}\int_{M} \phi^{2}e^{\Phi}\langle \nabla \varPhi, \nabla u\rangle u \mathrm{dvol}_{g(\tau)}(y)d\tau\\ 
	& -2 \int_{0}^{\hat{\tau}}\int_{M} \phi^{2}e^{\Phi}|\nabla u|^{2}\mathrm{dvol}_{g(\tau)}(y)d\tau + \int_{0}^{\hat{\tau}}\int_{M} \phi^{2}e^{\Phi}\partial_{s}\Phi u^{2}\mathrm{dvol}_{g(\tau)}(y)d\tau  \\ 
	& +\int_{M}\phi^{2}e^{\Phi}u^{2}|_{s = \hat{\tau}}\mathrm{dvol}_{g(\tau)}(y) - \int_{M}\phi^{2}e^{\Phi}u^{2}|_{s = 0}\mathrm{dvol}_{g(\tau)}(y).
	\end{split}
	\end{equation} 
	By the Cauchy-Schwarz inequality, 
	\begin{equation*}
	- 2 \int_{0}^{\hat{\tau}}\int_{M} \phi^{2}e^{\Phi}\langle \nabla \varPhi, \nabla u\rangle u \mathrm{dvol}_{g(\tau)}(y)d\tau\leq \int_{0}^{\hat{\tau}}\int_{M}\bigg( \phi^{2}e^{\Phi}|\nabla u|^{2} + \frac{1}{2} \phi^{2}e^{\Phi}u^{2}|\nabla \Phi|^{2}\bigg)\mathrm{dvol}_{g(\tau)}(y)d\tau. 
	\end{equation*}
	So, combining \eqref{0 integral} with the differential equation for, $\varPhi,$ and that, $u_{0}(x)= 0,$ we obtain that 
	\begin{equation*}
	4\int_{0}^{\hat{\tau}}\int_{M} \phi^{2}e^{\Phi}\langle \nabla \varPhi, \nabla u\rangle u\mathrm{dvol}_{g(\tau)}(y) \leq \int_{M}\phi^{2}e^{\Phi}u^{2}|_{s = \hat{\tau}}\mathrm{dvol}_{g(\tau)}(y). 
	\end{equation*}
	Now, $\phi = 1,$ on $B(x,r)$ and so
	\begin{equation}
	\begin{split}
	\int_{B(x,r)} e^{\varPhi}u^{2}|_{s = \hat{\tau}}\mathrm{dvol}_{g(\tau)}(y) & \leq \int \phi^{2}e^{\varPhi}u^{2}|_{s=\hat{\tau}}\mathrm{dvol}_{g(\tau)}(y) \\ 
	& \leq \frac{36}{\epsilon^{2}}\int_{0}^{\hat{\tau}} \int_{B(x,r + \epsilon)\setminus B(x,r)} e^{\varPhi}u^{2}\mathrm{dvol}_{g(\tau)}(y)d\tau. \label{exponential integral bound}
	\end{split}
	\end{equation}  
	By the definition of, $\varPhi,$ we have that 
	\begin{equation*}
	e^{-\frac{1}{16}} \leq e^{\varPhi}|_{B(x, \sqrt{\bar{\tau}/4}) \times [0,\hat{\tau}]}, \ \ \text{and} \ \ e^{\varPhi}|_{B(x, r + \epsilon)\setminus B(x,r) \times [0,\hat{\tau}]} \leq e^{-\frac{r^{2}}{8 \bar{\tau}}}.
	\end{equation*}
	Choosing, $r > \sqrt{\bar{\tau}/4},$ \eqref{exponential integral bound} yields that for all, $0 \leq \tau \leq \bar{\tau},$ 
	\begin{equation}
	e^{-\frac{1}{16}} \int_{B(x, \sqrt{\bar{\tau}/4})} u^{2}(y,\tau) \ \mathrm{dvol}_{g(\tau)}(y) \leq \frac{36}{\epsilon^{2}}e^{-\frac{r^{2}}{8 \tau}} \int_{0}^{\bar{\tau}} \int_{B(x,r + \epsilon)\setminus B(x,r)} u^{2}(y,\tau) \  \mathrm{dvol}_{g(\tau)}(y)d\tau.
	\end{equation}  
	By Moser's iteration, if we denote the $L^{2}$ Sobolev inequality on the ball, $B(x, \sqrt{\bar{\tau}}),$ by $C_{S}(x,\bar{\tau}),$ then 
	\begin{equation*}
	|u(x,\bar{\tau})|^{2} \leq C_{S}(x,\bar{\tau})\bar{\tau}^{-\frac{n+2}{2}}\int_{0}^{\bar{\tau}}\int_{B(x, \sqrt{\bar{\tau}})} u^{2}(y,\tau)\mathrm{dvol}_{g(\tau)}(y)d\tau. 
	\end{equation*}
	Integrating \eqref{exponential integral bound} over $[0,T]$ using the above inequality and recalling that \eqref{Bishop-Gromov} implies exponential growth of geodesic balls, $V_{g(\tau)} \leq e^{Ar^{2}},$ for some positive constant $A,$ we deduce that, 
	\begin{equation*}
	\begin{split}
	|u(x,\bar{\tau})|^{2} & \leq C_{S}(x,\bar{\tau})\bar{\tau}^{-\frac{n}{2}}e^{(\frac{1}{16}-\frac{r^{2}}{8 s})}\int_{0}^{\bar{\tau}}\int_{B(x,r + \epsilon)\setminus B(x,r)} u^{2}(y,\tau)\mathrm{dvol}_{g(\tau)}(y)d\tau \\ 
	& \leq \frac{36}{\epsilon^{2}}C_{S}(x,\bar{\tau})\bar{\tau}^{\frac{2-n}{2}}e^{(\frac{1}{16}-\frac{r^{2}}{8 s})}||u||^{2}_{L^{\infty}}\big(V_{g(s)}(B(x,r + \epsilon) - V_{g(s)}(B(x,r) \big) \\ 
	& \leq \frac{36}{\epsilon^{2}}C_{S}(x,\bar{\tau})\bar{\tau}^{\frac{2-n}{2}}||u||^{2}_{L^{\infty}}e^{(\frac{1}{16} + A(r+\epsilon)^{2} -\frac{r^{2}}{8s}}.
	\end{split}
	\end{equation*}
	So choose, $\bar{\tau} < (8A)^{-1},$ and let $r \mapsto \infty,$ then we conclude that for all, $0 < \bar{\tau} \leq (8A)^{-1},$  
	\begin{equation*}
	u_{0}(x) = 0.
	\end{equation*} 
	Then, use the semigroup property of the conjugate heat kernel to show that $u_{0}(x) = 0,$ for arbitrary $\bar{\tau},$ i.e. $0 < \bar{\tau} \leq k(8A)^{-1},$ where, $k,$ is a positive integer.
\end{proof} 
The uniqueness of the conjugate heat kernel allows us to both define the $\mathcal{W}$ functional and to show that this functional has a minimiser for non-compact $(M,g,f).$  
\begin{equation*}
\mathcal{W}(f,g,\tau) = \bigg\{\int_{M}\big[\tau(|\nabla f|^{2} + \mathrm{R}) + f - n \big] \ u\text{dvol}_{g(\tau)} \ : \int_{M} u\text{dvol}_{g(\tau)} = 1 \bigg\}, \ \mu(g,\tau) =\inf_{f} \mathcal{W}(f,g,\tau).
\end{equation*} 
By \eqref{trace} and the integral equality, $\int |\nabla f|^{2} u\mathrm{dvol} = \int \Delta f u \mathrm{dvol},$ we see that 
\begin{equation} \label{easy W}
\mathcal{W}(f,g,\tau) = \bigg\{\int_{M}\bigg(f - \frac{n}{2} \bigg) \ u\text{dvol}_{g(\tau)} \ : \int_{M} u\text{dvol}_{g(\tau)} = 1 \bigg\}, \ \mu(g,\tau) =\inf_{f} \mathcal{W}(f,g,\tau).
\end{equation}
Furthermore, as 
\begin{equation*}
\begin{split}
\partial_{\tau}\mathcal{W}(f,g,\tau) & = -2\tau\ \int_{N} \bigg|\mathrm{Ric} + \nabla^{2}f - \frac{g}{2\tau}\bigg|^{2}  \ u\text{dvol}_{g(\tau)} \\ 
& = 0,
\end{split}
\end{equation*}
the $\mathcal{W}$ functional is time-independent on $(M,g,f).$ Hence above quantities are time-independent.\newline
We now have the necessary components to show that $(M,g,f)$ is $\kappa$ non-collapsed at finite scales. 
\begin{theorem}\label{kappa non-collapsed}
On $(M,g,f)$ we have 
\begin{equation}
 \mathrm{vol}_{g(\tau)}(B(x,r)) \geq \kappa r^{n}.
\end{equation}
\end{theorem} 
The uniqueness of the conjugate heat kernel means that the proof of the above Theorem is identical to \cite{RIC} Theorem 4.1. More details of this proof can be found in \cite{KL} Theorem 13.3. Note that the Theorem does not require the quadratic decay on the scalar curvature, which follows from \eqref{easy W}. Also note that the minimiser is time-independent and hence one does not have the constrain of the Theorem only holding for finite times as in \cite{RIC} Theorem 4.1. In addition, by the parabolic scaling of the conjugate heat equation and the uniqueness of the conjugate heat kernel, $\frac{{r}^{2}}{\tau},$ is always bounded. Note that a similar proof was given - under the 'assumption' \footnote{The authors do not actually work with the conjugate heat kernel.} of conjugate heat kernel uniqueness - in \cite{HM} Lemma 2.3; see the Appendix.  \\ 

We now turn to the classification of four dimensional shrinking Ricci solitons. Recall \cite{HM2} Theorem 1.1 - which is an improvement on \cite{HM} Theorem 1.2; 
\begin{theorem}
Let $(M^{4}_{i}, g_{i}, f_{i})$ be a sequence of shrinking Ricci solitons with a finite, uniform lower bound on the $\mathcal{W}$ fucntional. Then for a point $p_{i},$ a subsequence of $(M^{4}_{i}, g_{i}, f_{i}, p_{i})$ converges to a unique, smooth Orbifold shrinking Ricci soliton, where the convergence is in the pointed Cheeger-Gromov sense.
\end{theorem}
\begin{definition}
A $C^{k}$ Orbifold $(\mathcal{O},g_{1})$ is a topological space which is a smooth Riemannian manifold away from a set of finitely many singular points. At a singular point $p,$ $\mathcal{O}$ is locally diffeomorphic to a cone over a spherical space form, $\mathbb{S}^{n-1}/\Gamma, \ \Gamma \subset SO(n).$ Further, at a singular point $p,$ the metric is locally the quotient of a $C^{k}, \ \Gamma$-invariant Riemannian metric on the Euclidean ball minus the origin.
\end{definition} 
The uniqueness of the conjugate heat kernel allows us to remove the assumption on the $\mathcal{W}$ functional in \cite{HM2} Theorem 1.1. In addition, the group $SO(4),$ has been classified in \cite{Wolf} Chapters 7.4 and 7.5. Thus we arrive at the following topological classification of 4D shrinking Ricci solitons, which are the singularity models of Type I solutions to the Ricci flow - see Introduction.
\begin{corollary}\label{4D classification}
The singularity models of Type I solutions to the Ricci flow in four dimensions are topologically
\begin{equation*}
\mathbb{R}^{4}/SO(4).
\end{equation*}
\end{corollary}
Note that \cite{HM2} Theorem 1.1 implies 
\begin{equation*}
\int_{B(p,r)}|\mathrm{Rm}|^{2} \leq C(\mu(g), r).
\end{equation*}
Hence one can find an $r$ small enough - or one could make an assumption on the smallnes of $\mu(g)$ - to apply the $\varepsilon$ Regularity Theorem, i.e. 
\begin{equation}\label{quadratic decay}
\sup_{B(x,r/2)} |\mathrm{Rm}| \leq \frac{c(r,n,\mu(g))}{r^{2}}.
\end{equation}
See \cite{HM} Lemma 3.3 for details. \newline 
Theorem \ref{kappa non-collapsed} and \eqref{quadratic decay} imply an injectivity radius lower bound for small balls;
\begin{equation*}
\mathrm{inj} \geq \eta. r >0.
\end{equation*}
See \cite{CGT} Theorem 4.7. As we view $(M^{4},g,f)$ from a dynamic perspective, we also have Shi's local derivative estimates on this small ball, $B_{g(\tau)}(x,r/2), \ \tau \in [0, T];$
\begin{equation}\label{Shi} 
	|D^{k}\mathrm{Rm}| \leq \frac{C_{1}}{r^{2+k}}, \ |D^{k}\mathrm{Rm}| \leq \frac{C_{2}}{\tau^{k}}, \ |\partial^{j}_{\tau}D^{k}\mathrm{Rm}| \leq \frac{C_{3}}{\tau^{j+k}}, \ \forall j,k \geq 0,
\end{equation}
where $C_{i} = C_{i}(n,r,\mu(g),j,k,T).$ See \cite{formation} Theorem 7.1 and Corollary 7.2 or \cite{KL} Theorem D.1 for the local version that we have used.\\ 

Let us point out some further important aspects which arise from Theorem \ref{uniqueness}.\newline 
Note that the no-local collapsing estimate allows us to prove a Gaussian upper bound on the conjugate heat kernel, given by.
\begin{equation}
	u(x,y,\tau) \leq c_{1}(4\pi \tau)^{-n/2}e^{-d^{2}_{g(0)}/4\tau}, \ c_{1} = c_{1}(n) >0. \label{HKUB}. 
\end{equation}
The method to obtain the upper bound is almost identical as the proof of Theorem \ref{uniqueness} and can be found in \cite{LY} Lemma 3.2 or \cite{CLY} Theorem 3.
In summation, the conjugate heat kernel is such that; 
\begin{corollary}
On $(M,g,f),$ the conjugate heat kernel, $u,$ has the two-sided Gaussian bounds, 
\begin{equation}\label{Gaussian bounds}
(4\pi\tau)^{-n/2}e^{-d^{2}_{g(\tau)}(x,y)/4\tau} \leq u(x,y,\tau) \leq c_{1} \leq (4\pi\tau)^{-n/2}e^{-d^{2}_{g(\tau)}(x,y)/4\tau}.
\end{equation} 
\end{corollary} 
The above corollary allows us to prove a short time Laplacian distance comparison, without any curvature assumptions.
\begin{corollary}\label{Varadhan-Laplace}
On $(M,g,f)$ 
\begin{equation}
\Delta d^{2}_{g(0)}(x,y) \leq 2n.
\end{equation}
\end{corollary}
\begin{proof}
By \eqref{Gaussian bounds} the asymptotic formula \eqref{Varadhan} holds, i.e. 
\begin{equation*}
\lim_{\tau \mapsto 0} 4\tau \log u(x,y,\tau) = - d^{2}_{g(0)}(x,y).
\end{equation*}
As $\log u = - f - \frac{n}{2}\log(4\pi \tau),$ and the scalar curvature is non-negative, \eqref{soliton equation} implies that 
\begin{equation*}
-\Delta \log u \leq \frac {n}{2\tau}.
\end{equation*}
The conclusion then follows.
\end{proof}
Finally, we give an analytic condition which is equivalent to the compactness of $(M,g,f).$\footnote{We have since learned that such a Theorem has been proven in the case of non-negative Ricci where one has a linear volume growth lower bound. See \cite{Bueler} Theorem 6.4.} 
\begin{lemma}
	$(M,g,f)$ is compact if and only if 
	\begin{equation}\label{compact}
	\int_{M}u^{-1} \ \mathrm{dvol}_{g(\tau)} < \infty. 
	\end{equation}
\end{lemma}
\begin{proof}
	Suppose that $(M,g,f)$ is compact. Then the Gaussian bounds imply that 
	\begin{equation*}
	\int_{M}u^{-1} \ \text{dvol}_{g(\tau)} \leq (4\pi \tau)^{n/2}\int_{M} e^{d^{2}_{g(\tau)}(x,y)/4\tau} \ \text{dvol}_{g(\tau)} < \infty. 
	\end{equation*}
	On the other hand, $(M,g,f)$ has infinite volume unless it is compact. So if \eqref{compact} is finite, $(M,g,f)$ must be compact. 
\end{proof}

\end{document}